\documentclass[USenglish]{article}

\usepackage[utf8]{inputenc}
\usepackage[small]{dgruyter}
\usepackage{microtype}
\usepackage{tikz}
\usepackage{tikz-qtree}
\usetikzlibrary{positioning,babel}

\newcommand{\Fqn}{\mathbb{F}_{q^n}}
\newcommand{\Fq}{\mathbb{F}_q}

\newcommand{\Mod}[1]{\ \mathrm{mod}\ #1}

\theoremstyle{dgdef}
\newtheorem{Theorem}{Theorem}[section]\newtheorem{Remark}{Remark}[section]\newtheorem{Lemma}{Lemma}[section]
\newtheorem{Corollary}{Corollary}[section]
\newtheorem{Definition}{Definition}[section]

\begin{document}

  \articletype{Research Article}
  \author[1]{Gary McGuire}
  \author[2]{Daniela Mueller}
  \affil[1]{School of Mathematics and Statistics, University College Dublin, Ireland}
  \affil[2]{School of Mathematics and Statistics, University College Dublin, Ireland}
  \title{Some Results on Linearized Trinomials that Split Completely}
  \runningtitle{...}
  \abstract{Linearized polynomials over finite fields have been much studied over the last several decades.
Recently there has been a renewed interest in linearized polynomials because of new connections to coding theory and finite geometry.
We consider the problem of calculating the rank or nullity of a linearized polynomial $L(x)=\sum_{i=0}^{d}a_i x^{q^i}$ (where $a_i\in \mathbb{F}_{q^n}$) from the coefficients $a_i$.
The rank and nullity of $L(x)$ are the rank and nullity of the associated $\mathbb{F}_q$-linear map $\mathbb{F}_{q^n} \longrightarrow \mathbb{F}_{q^n}$. McGuire and Sheekey \cite{MCGUIRE201968} defined a $d\times d$ matrix $A_L$ with the property that $$\mbox{nullity} (L)=\mbox{nullity} (A_L -I).$$ We present some consequences of this result for some trinomials that split completely, i.e., trinomials $L(x)=x^{q^d}-bx^q-ax$ that have nullity $d$. We give a full characterization of these trinomials for $n\le d^2-d+1$.}
  \keywords{Linearized Polynomials, Finite Field, ECDLP, elliptic curves, cryptography}
  \received{...}
  \accepted{...}
  \journalname{...}
  \journalyear{...}
  \journalvolume{..}
  \journalissue{..}
  \startpage{1}
  \aop
  \DOI{...}

\maketitle

\section{Introduction}

Let $\Fqn$ be the finite field with $q^n$ elements, where $q$ is a prime power. 
Let $$L(x)=a_0x+a_1x^q+a_2 x^{q^2}+\cdots+a_d x^{q^d}$$
be a $q$-linearized polynomial with coefficients in $\Fqn$. The roots of $L(x)$ that lie in the field $\Fqn$ form an $\Fq$-vector space, which can have dimension anywhere between 0 and $d$.

The dimension of the space of roots of $L$ that lie in  $\Fqn$ is equal to the nullity of $L$ considered as an $\Fq$-linear map from  $\Fqn$ to  $\Fqn$. McGuire and Sheekey \cite{MCGUIRE201968} defined a $d\times d$ matrix $A_L$ with the property that $$\mbox{nullity} (L)=\mbox{nullity} (A_L -I_d).$$  The entries of $A_L$ can be computed directly from the coefficients of $L$.

In this paper we focus on the case of largest possible nullity, i.e., the case that $L(x)$ has all its roots in $\Fqn$. In this case, $\mbox{nullity}(L)=d$, and so $A_L-I_d$ has rank 0 and is therefore the zero matrix.
Thus we will be studying when $A_L=I_d$. This case of largest possible nullity was also obtained in  \cite{CMPZu}.

We also restrict to trinomials. When computing the rank or nullity, we may assume without loss of generality that $L(x)$ is monic. We will study polynomials of the form $$L(x)=x^{q^d}-bx^q-ax\in\mathbb{F}_{q^n}[x]$$ where $q$ is a prime power and $n\geq1$. We want to find $a,b\in\mathbb{F}_{q^n}$ such that $L$ splits completely over $\mathbb{F}_{q^n}$, i.e., $L$ has $q^d$ roots in $\mathbb{F}_{q^n}$. Thus, the problem becomes finding $a,b\in\mathbb{F}_{q^n}$ such that $A_L=I_d$. We will provide a full characterization of this situation for $n\le d(d-1)+1$. Our  results are summarized and stated in the following theorem. 

\begin{Theorem}\label{all}
\begin{enumerate}
\item If $n\leq(d-1)d$ and $d$ does not divide $n$, then there is no polynomial $L=x^{q^d}-bx^q-ax$ with $a,b\in\mathbb{F}_{q^n}$ that splits completely over $\mathbb{F}_{q^n}$. \item Let $n=id$ with $i\in\{1,\dots,d-1\}$. Let $L=x^{q^d}-bx^q-ax\in\mathbb{F}_{q^n}[x]$. Then $L$ has $q^d$ roots in $\mathbb{F}_{q^n}$ if and only if $a^{1+q^d+\dots+q^{(i-1)d}}=1$ and $b=0$.
\item Let $n=(d-1)d+1$. Let $L=x^{q^d}-bx^q-ax\in\mathbb{F}_{q^n}[x]$. Then $L$ has $q^d$ roots in $\mathbb{F}_{q^n}$ if and only if all the following hold:

\hskip1in $\bullet$ $N(a)=(-1)^{d-1}$

\hskip1in $\bullet$ $b=-a^{qe_1}$ where $e_1=\sum_{i=0}^{d-1} q^{id}$

\hskip1in $\bullet$   $d-1$ is a power of the characteristic of $\mathbb{F}_{q^n}$\\
where $N(a)=a^{1+q+\dots+q^{(d-1)d}}=a^{(q^n-1)/(q-1)}$.
\end{enumerate}
\end{Theorem}

We will prove part 1 in Section \ref{nless}, part 2 in Section \ref{ddivn} and part 3 in Sections \ref{nexact} and \ref{converse}. In Section \ref{crypt} we present a possible application to elliptic curve cryptography.
 
Our result generalizes a a result of Csajbok et al \cite{CMPZh} which states that $a_0x+a_1x^q+a_3x^{q^3}$ (where $a_i\in \mathbb{F}_{q^7}$) cannot have $q^3$ roots in $ \mathbb{F}_{q^7}$ if $q$ is odd. This is the $d=3$ case of our theorem. Also in that paper, the authors give one example of a trinomial that does split completely when $d=3$, $n=7$, and $q=2$. Our theorem characterizes fully the trinomials that split completely and allows us to count their number (for each nonzero $a$ of norm 1 there is one polynomial, so there are $\frac{q^n-1}{q-1}$ such trinomials).

One can trivially obtain some results by taking $q$-th powers. For example, when $n=2d-2$, the trinomial $x^{q^d}-bx^q-ax$ cannot have $q^d$ roots in $\Fqn$. This follows by taking the $q^{d-2}$ power of the trinomial. Our theorem extends this to a larger range of values of $n$. 
 
One recent application of calculating the rank of linearized polynomials concerns rank metric codes
and MRD codes, see \cite{Sheekey}. 
In particular, we would obtain an $\Fqn$-linear MRD code from a space of linearized polynomials of dimension $kn$ over $\Fq$, with the property that every nonzero element has rank at least $n-k+1$. For example, in the case $k=3$, we would obtain an MRD code from the set of all trinomials $cx^{q^d}-bx^q-ax$ ($a,b,c\in\mathbb{F}_{q^n}$) if all of them have nullity 0 or 1 or 2.
 
Finally, we set the scene for our results. We are seeking $n\geq1$ and $a,b\in\mathbb{F}_{q^n}$ such that $L=x^{q^d}-bx^q-ax$ splits over $\mathbb{F}_{q^n}$. The companion matrix $C_L$ of $L(x)=x^{q^d}-bx^q-ax$ as defined in \cite{MCGUIRE201968} is the $d\times d$ matrix
$$
\begin{pmatrix}
 0 & 0 & \ldots & 0 & a\\
 1 & 0 & \ldots & 0 & b\\
 0 & 1 & \ldots & 0 & 0\\
 \vdots & \vdots & \ddots & \vdots & \vdots\\
 0 & 0 & \ldots & 1 & 0\\
\end{pmatrix}.
$$
We define $A_L=A_{L,n}=C_LC_L^q\cdots C_L^{q^{n-1}}$, where $C^q$ means raising every matrix entry to the power of $q$. As stated above, $L$ splits completely over $\mathbb{F}_{q^n}$ if and only if $A_L=I_d$.

\section{Fixed $d$ not dividing $n$ and $n\leq(d-1)d$}\label{nless}

In this section we will prove the first part of Theorem \ref{all}.

\begin{Theorem}\label{small_n}
If $n\leq(d-1)d$ and $d$ does not divide $n$, then there is no polynomial $L=x^{q^d}-bx^q-ax$ with $a,b\in\mathbb{F}_{q^n}$ that splits completely over $\mathbb{F}_{q^n}$.
\end{Theorem}
\begin{proof}
We will write $A_{n}$ instead of $A_{L,n}$ as $L$ is fixed throughout the proof.

If $n=1$ then $A_{1}=C_L\neq I_d$. Indeed, if $n\leq d-1$ then the $(1,1)$ entry of $A_{n}$ is $0$, so $A_{n}\neq I_d$.

Note that $A_{n}=A_{n-1}C_L^{q^{n-1}}$. But the 1st column of $C_L^{q^{n-1}}$ is 
$\begin{pmatrix}0 1 0 \dots 0 \end{pmatrix}^T$. 
Thus, the $(1,1)$ entry of $A_{n}$ is the $(1,2)$ entry of $A_{n-1}$. If $n\geq d$ then the $(1,1)$ entry of $A_{n}$ is also the $(1,d)$ entry of $A_{n-d+1}$.

Let $M_k$ denote the $(1,d)$ entry of $A_{k}$. Then $M_1=a$, and $M_k=0$ for $k=2,\dots,d-1$, since $M_k=\begin{pmatrix}
0 \dots 0 M_1 \dots M_{k-1}
\end{pmatrix}\cdot\begin{pmatrix}
a^{q^{k-1}} b^{q^{k-1}} 0 \dots 0
\end{pmatrix}^T$.

Set $M_0=0$. Then for $k\geq d$, we have a recursive formula, which follows directly from matrix multiplication: 
\begin{equation}\label{recursion}
 M_k=M_{k-d}a^{q^{k-1}}+M_{k-d+1}b^{q^{k-1}}.
\end{equation}

\underline{Claim:} $M_j=0$ for $j=id+2,\dots,(i+1)d-(i+1)$ and $i=0,\dots,d-3$.

\underline{Proof of Claim:} We prove the claim by induction on $i$. The base case $i=0$ was done above. Note that if $M_{k-d}=M_{k-d+1}=0$ then $M_k=0$. So if the claim is true for $i$, then we have $M_{id+2+d}=0,\dots,M_{(i+1)d-(i+1)+d-1}=0$, i.e. the claim is true for $i+1$. This completes the proof of the claim.

Note that when $i=d-3$, then $id+2=(i+1)d-(i+1)$, so the claim is not true for $i=d-2$.

For the remaining $n$ not divisible by $d$, we will show that the $(1,1)$ entry of $A_{n}$ cannot be 1 if the $(1,j)$ entry is 0 for some $j\in\{2,\dots,d\}$, and thus  $A_{n}$  cannot be the identity matrix. Note that the $(1,j)$ entry of $A_{n}$ is $M_{n-d+j}$.

For $i=1,\dots,d-2$, we have $M_{(i-1)d+2}=0$ and thus $$M_{id+1}=M_{(i-1)d+1}a^{q^{id}}.$$ Since $M_1=a$, we have $M_{id+1}=a^{1+q^d+\dots+q^{id}}$.  But $M_{id+1}$ is the $(1,(i+1)d+1-n)$ entry of $A_{n}$ for $n=id+1,\dots,(i+1)d-1$. If $A_{n}=I_d$ then the $(1,(i+1)d+1-n)$ entry must be $0$, so we must have $M_{id+1}=a^{1+q^d+\dots+q^{id}}=0$, and thus $a=0$. 

Recall  that the $(1,1)$ entry of $A_{n}$ is $M_{n-d+1}$. But $M_{n-d+1}$ must be either 0 or a power of $a$, since all initial values $M_0,\dots,M_{d-1}$ of the recursive formula are either $a$ or $0$. Therefore, if $a=0$, we have $M_{n-d+1}=0$ and so $A_{n}\neq I_d$.
\end{proof}

\begin{Remark}\label{A_id}
The proof is not valid when $d$ divides $n$. If $n=id$ with $i\in\{1,\dots,d-1\}$,  the $(1,1)$ entry of $A_{id}$ is $M_{(i-1)d+1}=a^{1+q^d+\dots+q^{(i-1)d}}$ for $i\geq1$, and so we have the equation $a^{1+q^d+\dots+q^{(i-1)d}}=1$ and cannot deduce that $a=0$.
\end{Remark}

The recursive formula \eqref{recursion} established in the proof of Theorem \ref{small_n} is valid in greater generality: Set $M_{l,l-d}=1$, and $M_{l,k}=0$ for $k\leq0$ and $k\ne l-d$. For $1\leq l\leq d$ and $k\geq1$, let
\begin{equation}\label{allrecursion}
 M_{l,k}=M_{l,k-d}a^{q^{k-1}}+M_{l,k-d+1}b^{q^{k-1}}.
\end{equation}
Then $M_{l,k}$ is the $(l,d)$ entry of $A_{L,k}$. Furthermore, the $(l,j)$ entry of $A_{L,k}$ is $M_{l,k-d+j}$.

\section{Fixed $d$  dividing $n$ and $n\leq(d-1)d$}\label{ddivn}

In the case that $d$ divides $n$, we have a solution, namely $a=1$ and $b=0$, i.e., the polynomial $x^{q^d}-x$ splits completely because $\mathbb{F}_{q^n}$ has a subfield $\mathbb{F}_{q^d}$. We now characterize exactly which polynomials split completely.

\begin{Theorem}\label{d divides n}
Let $n=id$ with $i\in\{1,\dots,d-1\}$. Let $L=x^{q^d}-bx^q-ax\in\mathbb{F}_{q^n}[x]$. Then $L$ has $q^d$ roots in $\mathbb{F}_{q^n}$ if and only if $a^{1+q^d+\dots+q^{(i-1)d}}=1$ and $b=0$.
\end{Theorem}

\begin{proof}
By Remark \ref{A_id}, if $L$ splits completely, we have $a^{1+q^d+\dots+q^{(i-1)d}}=1$. Now the $(1,d+1-i)$ entry of $A_{id}$ is $M_{1,i(d-1)+1}$. For $i=1$, this is $M_{1,d}=ab^{q^{d-1}}$. For $i\geq2$, we have $M_{1,i(d-1)+1}=M_{1,(i-1)d-(i-1)}a^{q^{i(d-1)}}+M_{1,(i-1)(d-1)+1}b^{q^{i(d-1)}}$. But by the claim in the proof of Theorem \ref{small_n}, we have $M_{1,(i-1)d-(i-1)}=0$ for $i=2,\dots,d-1$. Thus 
\begin{align*}
M_{1,i(d-1)+1}&=M_{1,(i-1)(d-1)+1}b^{q^{i(d-1)}}\\
&=M_{1,(i-2)(d-1)+1}b^{q^{(i-1)(d-1)}+q^{i(d-1)}}\\
&=\dots\\
&=ab^{q^{d-1}+q^{2(d-1)}+\cdots+q^{i(d-1)}}.
\end{align*}
But if $A_{id}=I_d$ then $M_{1,i(d-1)+1}=0$, and since $a\neq0$, we must have $b=0$.

To show the converse, assume that $a^{1+q^d+\dots+q^{(i-1)d}}=1$ and $b=0$. Then the $(l,1)$ entry of $A_{id}$ is 
\begin{align*}
M_{l,(i-1)d+1}&=M_{l,(i-2)d+1}a^{q^{(i-1)d}}\\
&=\dots\\
&=M_{l,1-d}a^{1+q^d+\cdots+q^{(i-1)d}}\\
&=M_{l,1-d}\\
&=
\begin{cases}
1 \text{ for } l=1,\\
0 \text{ for } l=2,\dots,d.
\end{cases}
\end{align*}
By \cite[Corollary 3.2]{CMPZu}, this implies that $A_{id}=I_d$.
\end{proof}

\section{Fixed $d$ and $n=(d-1)d+1$}\label{nexact}

In this section we will prove some preliminary results which are part of the 
proof of Theorem \ref{all} part 3.

\subsection{Assuming $L$ splits completely}

If $n=(d-1)d+1$, then, the $(1,j)$ entry of $A_{L,n}$ is $M_{1,(d-2)d+j+1}$ (where $j=1,\dots,d$). So to get $A_{L,n}=I_d$, the following system of equations has to be satisfied for $l=1,\dots,d$
\begin{equation}\label{system1}
\begin{cases}
M_{l,(d-2)d+l+1}=1\\
M_{l,(d-2)d+j+1}=0 \text{ for } j=1,\dots,l-1,l+1,\dots,d
\end{cases}
\end{equation}

\begin{Lemma}\label{a_11}
$\displaystyle{M_{1,(d-2)d+2}=ab^{e_2}}$ where
$e_2=\frac{q^{(d-1)d}-q^{d-1}}{q^{d-1}-1}$.
\end{Lemma}

\begin{proof}
By the recursive formula \eqref{allrecursion}, $M_{1,(d-2)d+2}=M_{1,(d-3)d+2}a^{q^{(d-2)d+1}}+M_{1,(d-3)d+3}b^{q^{(d-2)d+1}}$. But it follows from the claim in the proof of Theorem \ref{small_n} that  $M_{1,(d-j-1)d+j}=0$ for $j=2,\dots,d-1$. Thus 
\begin{align*}
M_{1,(d-2)d+2}&=M_{1,(d-3)d+3}b^{q^{(d-2)d+1}}=
M_{1,(d-4)d+4}b^{q^{(d-2)d+1}+q^{(d-3)d+2}}=\dots \\
&=M_{1,d} b^{q^{(d-2)d+1}+q^{(d-3)d+2}+\dots+q^{2d-2}}\\
&=ab^{q^{(d-2)d+1}+q^{(d-3)d+2}+\dots+q^{2d-2}+q^{d-1}}\\
&=ab^{\sum_{i=1}^{d-1}q^{i(d-1)}}\\  &=ab^{e_2}.
\end{align*}
\end{proof}

\begin{Lemma}\label{a_1d}
$M_{1,(d-1)d+1}=a^{e_1}+ab^{e_2+q^{(d-1)d}}$
where $e_1=\frac{q^{d^2}-1}{q^d-1}$ and $e_2=\frac{q^{(d-1)d}-q^{d-1}}{q^{d-1}-1}$.
\end{Lemma}

\begin{proof}
By the recursive formula \eqref{allrecursion}, 
\[
M_{1,(d-1)d+1}=M_{1,(d-2)d+1}a^{q^{(d-1)d}}+M_{1,(d-2)d+2}b^{q^{(d-1)d}}.
\]
 By Lemma \ref{a_11}, $M_{1,(d-2)d+2}=ab^{e_2}$. 
Also $M_{1,(d-2)d+1}=a^{1+q^d+\dots+q^{(d-2)d}}$ as established in the proof of Theorem \ref{small_n}.

Therefore 
\begin{align*}
M_{1,(d-1)d+1}&=a^{\sum_{i=0}^{d-1}q^{id}}+ab^{e_2+q^{(d-1)d}}\\
&=a^{e_1}+ab^{e_2+q^{(d-1)d}}.
\end{align*}
\end{proof}

\begin{Theorem}\label{direction1}
Let $n=(d-1)d+1$. Let $L=x^{q^d}-bx^q-ax\in\mathbb{F}_{q^n}[x]$. If $L$ has $q^d$ roots in $\mathbb{F}_{q^n}$ then 
\begin{enumerate}
\item $a^{1+q+\dots+q^{(d-1)d}}=(-1)^{d-1} \mbox{ and } $
\item $a^{1+q e_1 e_2}=(-1)^{d-1} \mbox{ and }$
\item $b=-a^{qe_1},$
\end{enumerate}
where $e_1=\frac{q^{d^2}-1}{q^d-1}$ and $e_2=\frac{q^{(d-1)d}-q^{d-1}}{q^{d-1}-1}$.
\end{Theorem}

\begin{proof}
If $A_{L,n}=I_d$, then equation \eqref{system1} has to be satisfied. By Lemma \ref{a_11}, we have $ab^{e_2}=1$ (the $(1,1)$ entry of $A_{L,n}$), and by Lemma \ref{a_1d}, we have $a^{e_1}+ab^{e_2+q^{(d-1)d}}=0$ (the $(1,d)$ entry of $A_{L,n}$). But if $ab^{e_2}=1$, then $a^{e_1}+ab^{e_2+q^{(d-1)d}}=a^{e_1}+b^{q^{(d-1)d}}$, and thus we have $b^{q^{(d-1)d}}=-a^{e_1}$. Raising both sides to the power of $q$ gives us $b^{q^n}=(-1)^q a^{qe_1}$.  Since $q$ is a prime power, $(-1)^q=-1$ in $\mathbb{F}_{q^n}$. Thus, $b=-a^{qe_1}$ which proves the third conclusion.

Lemma \ref{a_11} says $ab^{e_2}=1$ which now implies $$a^{-1}=b^{e_2}=(-a^{qe_1})^{e_2}=(-1)^{e_2}a^{q e_1 e_2},$$ and so $a^{1+q e_1 e_2}=(-1)^{e_2}$. (Note that $a\ne0$ since $ab^{e_2}=1$.)

Recall that $e_2=\sum_{i=1}^{d-1}q^{i(d-1)}$. So if $q$ is even, then $e_2$ is even. If $q$ is odd, then $q^{i(d-1)}$ is odd for all $i=1,\dots,d-1$. So if $d-1$ is even, then $e_2$ is an even sum of odd numbers and thus even, and if $d-1$ is odd, then $e_2$ is an odd sum of odd numbers and thus odd. Thus, $(-1)^{e_2}=(-1)^{q(d-1)}$. Since $(-1)^q=-1$ in $\mathbb{F}_{q^n}$ we have $(-1)^{e_2}=(-1)^{d-1}$.

By \cite[Corollary 1]{MCGUIRE201968}, if $L$ splits, then $N(-a)=(-1)^{nd}N(1)$, where $N$ is the norm function over $\mathbb{F}_{q^n}$. So we have the additional condition $N(a)=(-1)^{n(d-1)}$ or $a^{\frac{q^n-1}{q-1}}=(-1)^{n(d-1)}$. But $n=(d-1)d+1$, so $n$ is always odd. Consequently, $(-1)^{n(d-1)}=(-1)^{d-1}$.

Hence, $a$ satisfies the equations
\begin{equation}\label{system2}
\begin{cases}
a^{1+q e_1 e_2}=(-1)^{d-1}\\
a^{\frac{q^n-1}{q-1}}=(-1)^{d-1}.
\end{cases}
\end{equation} 
\end{proof}

In the next section we will show that conclusion 1 of this theorem actually implies conclusion 2.

\subsection{GCD of $x^k\pm1$ and $x^l\pm1$}

The GCD of $x^k-1$ and $x^l-1$ is well known to be $x^{\gcd(k,l)}-1$, but we are interested in the GCD of $x^k+1$ and $x^l+1$. The following is surely well known, but we include a proof.

\begin{Theorem}\label{gcd}
The GCD of $x^k+1$ and $x^l+1$ is $x^{\gcd(k,l)}+1$ if $\frac{k}{\gcd(k,l)}$ and $\frac{l}{\gcd(k,l)}$ are both odd, and $1$ otherwise.
\end{Theorem}

\begin{proof}
Let $d=\gcd(k,l)$ and let $s,t$ be B\'{e}zout Coefficients for $k$ and $l$, i.e. $sk+tl=d$. Let $g=\gcd(x^k+1,x^l+1)$. Then $x^k\equiv-1 \Mod{g}$ and $x^l\equiv-1 \Mod{g}$. Thus $x^{sk+tl}\equiv(-1)^{s+t} \Mod{g}$. So $g$ divides $x^{sk+tl}-(-1)^{s+t}=x^d-(-1)^{s+t}$. We need to check if $x^d-(-1)^{s+t}$ divides $x^k+1$ and $x^l+1$.

Let $e=\frac{k}{d}$ and $f=\frac{l}{d}$. Then $x^k+1=x^{ed}+1=((-1)^{s+t})^e+1 \Mod{x^d-(-1)^{s+t}}$ and similarly, $x^l+1=((-1)^{s+t})^f+1 \Mod{x^d-(-1)^{s+t}}$. So we need to have $(-1)^{(s+t)e}+1=0$ and $(-1)^{(s+t)f}+1=0$, i.e. $e,f,s+t$ all need to be odd. But $sk+tl=d$ implies $se+tf=1$, so $e,f$ odd implies $s+t$ odd. Thus if $e,f$ are odd, then $g=x^d-(-1)^{s+t}=x^d+1$.
\end{proof}

\begin{Remark}
Similarly, one can show that $\gcd(x^k-1,x^l+1)=x^{\gcd(k,l)}+1$ if $\frac{k}{\gcd(k,l)}$ is even and $\frac{l}{\gcd(k,l)}$ is odd.
\end{Remark}

\begin{Lemma}\label{expos}
Let $n=(d-1)d+1$ and let $e_1=\frac{q^{d^2}-1}{q^d-1}$ and $e_2=\frac{q^{(d-1)d}-q^{d-1}}{q^{d-1}-1}$.
Then 
\[
\gcd(1+qe_1e_2,\frac{q^n-1}{q-1})=\frac{q^n-1}{q-1}.
\]
\end{Lemma}
\begin{proof}
We first show that $\frac{q^n-1}{q-1}=1+q(\frac{q^{d^2}-1}{q^d-1})(\frac{q^{(d-1)d}-q^{d-1}}{q^{d-1}-1}) \Mod{q^n-1}$. Recall that $\frac{q^n-1}{q-1}=\sum_{i=0}^{n-1}q^i=1+q+q^2+\dots+q^{n-1}$ and $n=d^2-d+1$. Then 
\begin{align*}
1+q(\frac{q^{d^2}-1}{q^d-1})(\frac{q^{(d-1)d}-q^{d-1}}{q^{d-1}-1})&=1+q(\sum_{i=0}^{d-1}q^{id})(\sum_{j=1}^{d-1}q^{j(d-1)})\\
&=1+\sum_{i=0}^{d-1}\sum_{j=1}^{d-1}q^{id+j(d-1)+1}.
\end{align*} 
We claim that $id+j(d-1)+1 \Mod{n}$ with $i=0,\dots,d-1$ and $j=1,\dots,d-1$ gives us exactly the numbers 
$\{1,\dots,n-1\}$. Assuming the truth of this claim, $1+q(\frac{q^{d^2}-1}{q^d-1})(\frac{q^{(d-1)d}-q^{d-1}}{q^{d-1}-1}) \Mod{q^n-1}=1+q+q^2+\dots+q^{n-1}=\frac{q^n-1}{q-1}$. Since $\frac{q^n-1}{q-1}$ divides $q^n-1$, $\frac{q^n-1}{q-1}$ divides $1+q(\frac{q^{d^2}-1}{q^d-1})(\frac{q^{(d-1)d}-q^{d-1}}{q^{d-1}-1})$ and thus $\gcd(1+q(\frac{q^{d^2}-1}{q^d-1})(\frac{q^{(d-1)d}-q^{d-1}}{q^{d-1}-1}),\frac{q^n-1}{q-1})=\frac{q^n-1}{q-1}$ and the result is proved.

It remains to prove the claim. To see this, we will show that the sets $$\{(i+j)d-(j-1)\mid i=0,\dots,d-1; j=1,\dots,d-1\}$$ and $$\{kd-m\mid m=0,\dots,d-2; k=m+1,\dots,m+d\}$$ are equal, and it is easy to see that all values in the second set are distinct.

Fixing $j$ and varying $i=0,\dots,d-1$ gives us the numbers $$jd-(j-1),(j+1)d-(j-1),\dots,(j+d-1)d-(j-1).$$
When $i+j\leq d-1$ then $(i+j)d-(j-1)\leq n$ and all these numbers are of the form $kd-m$ where $m\in\{0,\dots,d-2\}$ and $m<k\leq d-1$.\\
When $i+j\geq d$, then $(i+j)d-(j-1)>n$ and we subtract $n$ to get $(i+j-d+1)d-j$. Now $i+j-d+1\leq j$ since $i\leq d-1$, and thus $(i+j-d+1)d-j$ is not of the above form $kd-m$ with $m<k\leq d-1$.
\end{proof}

\begin{Corollary}
Let $n=(d-1)d+1$ and let $e_1=\frac{q^{d^2}-1}{q^d-1}$ and $e_2=\frac{q^{(d-1)d}-q^{d-1}}{q^{d-1}-1}$. Then  
\[
\gcd(x^{1+qe_1e_2}+(-1)^d,x^{\frac{q^n-1}{q-1}}+(-1)^d)=x^{\frac{q^n-1}{q-1}}+(-1)^d.
\]
\end{Corollary}

\begin{proof}
If $q$ is even, then both $1+q e_1 e_2$ and $\frac{q^n-1}{q-1}$ are odd. Recall that $\frac{q^n-1}{q-1}=\sum_{i=0}^{n-1}q^i$. So if $q$ is odd, then $\frac{q^n-1}{q-1}$ is odd if $n$ is odd, and even if $n$ is even. But $n=(d-1)d+1$ is always odd, so $\frac{q^n-1}{q-1}$ is odd. We have already established in the proof of Theorem \ref{direction1} that if $q$ is odd, then $e_2$ is odd if $d-1$ is odd, and even if $d-1$ is even.
Now $e_1=\sum_{i=0}^{d-1}q^{id}$ is odd if $q$ and $d$ are odd and even if $q$ is odd but $d$ is even. But either $d$ or $d-1$ is always even, so $e_1e_2$ is even. Thus $1+q e_1 e_2$ is odd. Consequently, by Theorem \ref{gcd} $$\gcd(x^{1+q e_1 e_2}+(-1)^d,x^{\frac{q^n-1}{q-1}}+(-1)^d)=x^{\gcd(1+q e_1 e_2,\frac{q^n-1}{q-1})}+(-1)^d$$ for any $q,d$.
\end{proof}

\begin{Corollary}\label{direction1_2}
In the conclusions of Theorem \ref{direction1}, conclusion 1 implies conclusion 2.
\end{Corollary}

\section{The Main Result}\label{converse}

In this section, we will prove the third part of the main theorem as stated in the introduction.
The following Lemma is surely well known but we include a short proof.

\begin{Lemma}\label{binomodp}
$\binom{n}{i}=0 \Mod{p}$ for all $i=1,2,\dots,n-1$ if and only if $n$ is a power of $p$.
\end{Lemma}
\begin{proof}
If $n$ is a power of $p$, then the above binomial coefficients are divisible by $p$. On the other hand, we claim that if $n=p^kw$, where $p\nmid w$, $w>1$ and $k\geq0$, then $\binom{p^kw}{p^k}$ is not divisible by $p$. First note that $\binom{n}{m}=\frac{n!}{m!(n-m)!}=\frac{\prod_{i=1}^n i}{(\prod_{i=1}^m i)(\prod_{i=1}^{n-m}i)}=\frac{\prod_{i=n-m+1}^n i}{\prod_{i=1}^m i}=\frac{\prod_{i=0}^{m-1}(n-i)}{\prod_{i=1}^m i}=\frac{n}{m}\prod_{i=1}^{m-1}\frac{n-i}{i}$. Thus $\binom{p^kw}{p^k}=w\prod_{i=1}^{p^k-1}\frac{p^kw-i}{i}$. Now write $i=lp^j$ with $p\nmid l$. Then $\frac{p^kw-i}{i}=\frac{p^kw-lp^j}{lp^j}=\frac{(p^{k-j}w-l)p^j}{lp^j}=\frac{p^{k-j}w-l}{l}$ which is not divisible by $p$.
\end{proof}

Finally, we present  the last part of the proof of the main theorem.

\begin{Theorem}\label{gen_ab}
Let $n=(d-1)d+1$ and $e_1=\frac{q^{d^2}-1}{q^d-1}$. Let $L=x^{q^d}-bx^q-ax\in\mathbb{F}_{q^n}[x]$. Then $L$ has $q^d$ roots in $\mathbb{F}_{q^n}$ if and only if each of the following holds:
\begin{enumerate}
\item $a^{1+q+\dots+q^{(d-1)d}}=(-1)^{d-1}$
\item $b=-a^{qe_1}$
\item $d-1$ is a power of the characteristic of $\mathbb{F}_{q^n}$.
\end{enumerate}
\end{Theorem}

\begin{proof}
Recall that the $(l,1)$ entry of $A_{L,n}$ is $M_{l,n-d+1}$. We will first show that 
\[
M_{l,n-d+1}=\begin{cases}
1 \text{ for } l=1,\\
0 \text{ for } l=2,\dots,d
\end{cases}
\] 
whenever the three conditions of the theorem are fulfilled. By \cite[Corollary 3.2]{CMPZu}, this implies that $A_{L,n}=I_d$.\\

Let $k\geq d+1$. By the  recursion \eqref{allrecursion}, 
\begin{align}
M_{l,k}&=M_{l,k-d}a^{q^{k-1}}+M_{l,k-d+1}b^{q^{k-1}}\nonumber \\
&=(M_{l,k-2d}a^{q^{k-d-1}}+M_{l,k-2d+1}b^{q^{k-d-1}})a^{q^{k-1}}+\nonumber \\
&\qquad (M_{l,k-2d+1}a^{q^{k-d}}+M_{l,k-2d+2}b^{q^{k-d}})b^{q^{k-1}}\nonumber \\
&=M_{l,k-2d}a^{q^{k-d-1}+q^{k-1}}+M_{l,k-2d+1}(a^{q^{k-1}}b^{q^{k-d-1}}+a^{q^{k-d}}b^{q^{k-1}})
\nonumber \\
&\qquad +M_{l,k-2d+2}b^{q^{k-d}+q^{k-1}}.\label{hhh}
\end{align}
Since $b=-a^{qe_1}=-a^{1+\sum_{i=0}^{d-2} q^{id+1}} \Mod{a^{q^n}-a}$ (condition 2 in the statement of the theorem) we have
\begin{align*}
a^{q^{k-1}}b^{q^{k-d-1}}
&=-a^{q^{k-1}+q^{k-d-1}(1+\sum_{i=0}^{d-2} q^{id+1})}\\
&=-a^{q^{k-1}+q^{k-d}+q^{k-d-1}+\sum_{i=1}^{d-2} q^{k+(i-1)d}}\\
&=-a^{q^{k-1}+q^{k-d}+q^{k-d-1}+\sum_{i=0}^{d-3} q^{k+id}}\\
&=-a^{q^{k-1}+q^{k-d}+q^{k+(d-2)d}+\sum_{i=0}^{d-3} q^{k+id}} \Mod{a^{q^n}-a}\\
&=-a^{q^{k-d}+q^{k-1}(1+\sum_{i=0}^{d-2} q^{id+1})}\\
&=a^{q^{k-d}}b^{q^{k-1}}.
\end{align*}

So the coefficient of $M_{l,k-2d+1}$ in \eqref{hhh} that comes from expanding $M_{l,k-d}$ is the same as the coefficient that comes from expanding $M_{l,k-d+1}$.

\bigskip
\begin{center}
\Tree [.$M_{l,k}$  [.$M_{l,k-d}$  [.$M_{l,k-2d}$  ]  $M_{l,k-2d+1}$ ][.$M_{l,k-d+1}$  $M_{l,k-2d+1}$ $M_{l,k-2d+2}$ ] ]
\end{center}
\bigskip

Let $c_{2,0}=a^{q^{k-1}+q^{k-d-1}}$, $c_{2,1}=a^{q^{k-d}}b^{q^{k-1}}$, and $c_{2,2}=b^{q^{k-1}+q^{k-d}}$. Thus \eqref{hhh} is saying that $$M_{l,k}=c_{2,0}M_{l,k-2d}+2c_{2,1}M_{l,k-2d+1}+c_{2,2}M_{l,k-2d+2}.$$ One can see Pascal's triangle emerging. We claim that $$M_{l,k}=\sum_{i=0}^{j} \binom{j}{i} c_{j,i}M_{l,k-jd+i}$$ for all $j=0,\dots,\lfloor\frac{k-1}{d}+1\rfloor$, where $c_{j,i}$ are expressions in $a$ and $b$, determined by the following recursion: 
\[
c_{j,i} = 
\begin{cases}
1 & \text{for } j=i=0, \\
c_{j-1,0}a^{q^{k-(j-1)d-1}} & \text{for } i=0, \\
c_{j-1,i}a^{q^{k-(j-1)d+i-1}}=c_{j-1,i-1}b^{q^{k-(j-1)d+i-2}} & \text{for } 0<i<j,\\
c_{j-1,j-1}b^{q^{k-(j-1)d+j-2}}  & \text{for } i=j.
\end{cases}
\]

We have shown the statement for $j=2$. Assume that the statement is true for any index less than $j$. Then
\begin{align*}
M_{l,k}&=\sum_{i=0}^{j-1} \binom{j-1}{i} c_{j-1,i}M_{l,k-(j-1)d+i}\\
&=\sum_{i=0}^{j-1} \binom{j-1}{i} c_{j-1,i}(M_{l,k-jd+i}a^{q^{k-(j-1)d+i-1}}+M_{l,k-jd+i+1}b^{q^{k-(j-1)d+i-1}})\\
&=\binom{j-1}{0}c_{j-1,0}a^{q^{k-(j-1)d-1}}M_{l,k-jd}+\binom{j-1}{j-1}c_{j-1,j-1}b^{q^{k-(j-1)d+j-2}}M_{l,k-jd+j}\\
&+\sum_{i=1}^{j-1}M_{l,k-jd+i}\left(\binom{j-1}{i-1}c_{j-1,i-1}b^{q^{k-(j-1)d+i-2}}+\binom{j-1}{i}c_{j-1,i}a^{q^{k-(j-1)d+i-1}}\right).
\end{align*}
Let $m=k-(j-2)d+i-1$. Then $a^{q^{m-1}}b^{q^{m-d-1}}=a^{q^{m-d}}b^{q^{m-1}}$, i.e.
\begin{align*}
a^{q^{k-(j-2)d+i-2}}b^{q^{k-(j-1)d+i-2}}=a^{q^{k-(j-1)d+i-1}}b^{q^{k-(j-2)d+i-2}}
\end{align*}
and hence
\begin{align}\label{c_m}
c_{j-2,i-1}a^{q^{k-(j-2)d+i-2}}b^{q^{k-(j-1)d+i-2}}=c_{j-2,i-1}a^{q^{k-(j-1)d+i-1}}b^{q^{k-(j-2)d+i-2}}.
\end{align}
Then 
$$c_{j-1,i-1}b^{q^{k-(j-1)d+i-2}}=c_{j-2,i-1}a^{q^{k-(j-2)d+i-2}}b^{q^{k-(j-1)d+i-2}}$$ and $$c_{j-1,i}a^{q^{k-(j-1)d+i-1}}=c_{j-2,i-1}b^{q^{k-(j-2)d+i-2}}a^{q^{k-(j-1)d+i-1}}$$ and by \eqref{c_m}, these two expressions are equal.\\
Thus
\begin{align*}
M_{l,k}&=\binom{j}{0}c_{j-1,0}a^{q^{k-(j-1)d-1}}M_{l,k-jd}+\binom{j}{j}c_{j-1,j-1}b^{q^{k-(j-1)d+j-2}}M_{l,k-jd+j}\\
&+\sum_{i=1}^{j-1}\binom{j}{i}c_{j-1,i}a^{q^{k-(j-1)d+i-1}}M_{l,k-jd+i}
\end{align*} as desired. This completes the proof of the claim.

We now have
\begin{align*}
M_{l,n-d+1}&=M_{l,(d-2)d+2}=\sum_{i=0}^{d-1} \binom{d-1}{i} c_{d-1,i}M_{l,2-d+i}\\
&=\sum_{i=0}^{d-2} \binom{d-1}{i} c_{d-1,i}M_{l,2-d+i}+c_{d-1,d-1}M_{l,1}\\
&=\begin{cases}
c_{d-1,d-1}M_{l,1} & \text{for } l=1, \\
\binom{d-1}{l-2} c_{d-1,l-2}M_{l,l-d}+c_{d-1,d-1}M_{l,1} & \text{for } l\geq2\\
\end{cases}
\end{align*}
since $M_{l,2-d+i}=0$ when $i\neq l-2$.

As before, let $e_1=\frac{q^{d^2}-1}{q^d-1}$ and $e_2=\frac{q^{(d-1)d}-q^{d-1}}{q^{d-1}-1}$.

Then
\begin{align*}
c_{d-1,d-1}&=b^{q^{k-1}+q^{k-d}+q^{k-2d+1}+\dots+q^{k-(d-2)d+d-3}}\\
&=b^{q^{(d-2)d+1}+q^{(d-3)d+2}+\dots+q^{d-1}}\text{ for }k=(d-2)d+2\\
&=b^{q^{d-1}+q^{2d-2}+\dots+q^{(d-1)(d-1)}}\\
&=b^{e_2}\\ &=(-a^{qe_1})^{e_2}\\ &=(-1)^{d-1} a^{qe_1e_2}\\
&=(-1)^{d-1} a^{\frac{q^n-1}{q-1}-1} \Mod{a^{q^n}-a} \text{ by Lemma \ref{expos}}.
\end{align*}
Also
\begin{align*}
c_{d-1,0}&=a^{q^{k-1}+q^{k-d-1}+\dots+q^{k-(d-2)d-1}}\\
&=a^{q^{(d-2)d+1}+q^{(d-3)d+1}+\dots+q} \text{ for }k=(d-2)d+2.
\end{align*}
Thus \begin{align*}
c_{d-1,d-1}M_{l,1}&=(-1)^{d-1} a^{\frac{q^n-1}{q-1}-1}(aM_{l,1-d}+bM_{l,2-d})\\
&=(-1)^{d-1} a^{\frac{q^n-1}{q-1}}M_{l,1-d}+(-1)^{d} a^{\frac{q^n-1}{q-1}+\sum_{i=0}^{d-2} q^{id+1}}M_{l,2-d}\\
&=(-1)^{d-1} (-1)^{d-1}M_{l,1-d}+(-1)^{d} (-1)^{d-1}c_{d-1,0}M_{l,2-d}\\
&=M_{l,1-d}-c_{d-1,0}M_{l,2-d}
\end{align*}
since $a^{\frac{q^n-1}{q-1}}=(-1)^{d-1}$ (condition 1 in the statement of the theorem).

Hence, 
\begin{align}
M_{l,n-d+1}&=\begin{cases}
M_{l,1-d}-c_{d-1,0}M_{l,2-d} & \text{for } l=1, \\
\binom{d-1}{l-2} c_{d-1,l-2}M_{l,l-d}+M_{l,1-d}-c_{d-1,0}M_{l,2-d} & \text{for } l\geq2\\
\end{cases}\nonumber \\
&=
\begin{cases}
1 & \text{for } l=1, \\
0 & \text{for } l=2, \\
\binom{d-1}{l-2} c_{d-1,l-2} & \text{for } l\geq3  \label{jjj}
\end{cases}
\end{align} since $M_{l,l-d}=1$ and $M_{l,k}=0$ when $k\neq l-d$ and $k\leq0$.

So far we have only used conditions 1 and 2 in the statement of the theorem (so note for later that conditions 1 and 2 imply \eqref{jjj}). Assume now that condition 3 holds. By Lemma \ref{binomodp}, $M_{l,n-d+1}=0$ for all $l\geq3$ because $\binom{d-1}{l-2} =0$. This completes the proof that if the three conditions in the statement hold, then $L$ splits completely.

Now we complete the proof of the theorem by showing the converse, i.e. we show that if $L$ splits completely then the three conditions in the statement hold. Theorem \ref{direction1} and Corollary \ref{direction1_2} show that if $L$ splits completely, then conditions 1 and 2 of the theorem hold. Because conditions 1 and 2 hold, we know that \eqref{jjj} holds.

On the other hand, since $L$ splits completely, $M_{l,n-d+1}=0$ for all $l\geq3$. Therefore $\binom{d-1}{l-2} c_{d-1,l-2}=0$ for all $3 \leq l\leq d$. We now use the fact that $c_{d-1,l-2}$ is a power of $a$, and is therefore nonzero because $a$ is nonzero. We are forced to conclude that  $\binom{d-1}{l-2} =0$ for all $3 \leq l\leq d$. This implies that $d-1$ is a power of the characteristic of $\mathbb{F}_{q^n}$ by Lemma \ref{binomodp}.
\end{proof}

\section{Possible Application to Cryptography}\label{crypt}

\subsection{Quasi-Subfield Polynomials}

The recent work \cite{Quasi} explored the use of quasi-subfield polynomials to solve the Elliptic Curve Discrete Logarithm Problem (ECDLP). They define quasi-subfield polynomials as polynomials of the form $x^{q^d}-\lambda(x)\in\mathbb{F}_{q^n}[x]$ which divide $x^{q^n}-x$ and where 
$\log_q(\deg(\lambda))<d^2/n$. For appropriate choices of $n$ and $d$, linearized polynomials 
have a chance of being quasi-subfield polynomials. We first observe that the polynomials in  Theorem \ref{gen_ab} are quasi-subfield polynomials.

\begin{Lemma}\label{linqs} 
The linearized polynomial $L=x^{q^d}-bx^q-ax\in\mathbb{F}_{q^{(d-1)d+1}}[x]$ is a quasi-subfield polynomial when all the following conditions are satisfied.
 \begin{enumerate}
\item $a^{1+q+\dots+q^{(d-1)d}}=(-1)^{d-1}$
\item $b=-a^{qe_1}$
\item $d-1$ is a power of the characteristic of $\mathbb{F}_{q^n}$.
\end{enumerate}
\end{Lemma}

\begin{proof}
Here, $\log_q(\deg(\lambda))=1$ and $d^2>n=(d-1)d+1$ so the condition
$\log_q(\deg(\lambda))<d^2/n$ is satisfied.
By Theorem \ref{gen_ab}, $L(x)$ divides $x^{q^n}-x$.
\end{proof}

\subsection{The ECDLP}

Let $E$ be an elliptic curve over a finite field $\mathbb{F}_q$, where $q$ is a prime power. In practice, $q$ is often a prime number or a large power of 2. Let $P$ and $Q$ be 
$\mathbb{F}_q$-rational points on $E$. The Elliptic Curve Discrete Logarithm Problem (ECDLP) is finding an integer $l$ (if it exists) such that $Q=lP$. The integer $l$ is called the discrete logarithm of $Q$ to base $P$.

The ECDLP is a hard problem that underlies many cryptographic schemes and is thus an area of active research.  The introduction of summation polynomials by \cite{Semaev04} has led to algorithms that resemble the index calculus algorithm of the DLP over finite fields.

The algorithm to solve the ECDLP in \cite{Quasi} also uses summation polynomials, so we recall their definition.

\begin{Definition}\cite{Semaev04}
Let $E$ be an elliptic curve over a field $K$. For $m\geq 1$, we define the summation polynomial $S_{m+1}=S_{m+1}(X_0,X_1,\ldots,X_m)\in K[X_0,X_1,\ldots,X_m]$ of $E$ by the following property. Let $x_0,x_1,\ldots,x_m\in\overline{K}$, then $S_{m+1}(x_0,x_1,\ldots,x_m)=0$ if and only if there exist $y_0,y_1,\ldots,y_m\in\overline{K}$ such that $(x_i,y_i)\in E(\overline{K})$ and $(x_0,y_0)+(x_1,y_1)+\ldots+(x_m,y_m)=\mathcal{O}$, where $\mathcal{O}$ is the identity element of $E$.
\end{Definition}

The summation polynomials $S_m$ have many terms and have only been computed for $m\le 9$.

\cite{Quasi} develop an algorithm to solve the ECDLP over the field $\mathbb{F}_{q^n}$ using a quasi-subfield polynomial $X^{q^d}-\lambda(X)\in\mathbb{F}_{q^n}[X]$ and the summation polynomial $S_{m+1}(X_0,X_1,\ldots,X_m)\in\mathbb{F}_{q^n}[X_0,X_1,\ldots,X_m]$. By \cite[Theorem 3.2]{Quasi} (see also Appendix A1) their algorithm has complexity $$m!q^{n-d(m-1)}\tilde{O}(m^{5.188}2^{7.376m(m-1)}\deg(\lambda)^{4.876m(m-1)})+mq^{2d}.$$

\subsection{Linearized Quasi-Subfield Polynomials}

One of the problems outlined in \cite{Quasi} is to find suitable quasi-subfield polynomials that give optimal complexity in their algorithm. So in this section, we will investigate whether the linearized polynomials 
in this paper are a suitable choice. 

In our notation the field is  $\mathbb{F}_{q^n}$ so brute force algorithms have $O(q^n)$ complexity and generic algorithms (Pollard Rho or Baby-Step-Giant-Step) have $O(q^{n/2})$ complexity.

If $n=(d-1)d+1$ and we use $L=x^{q^d}-bx^q-ax\in\mathbb{F}_{q^{(d-1)d+1}}[x]$ 
as in Lemma \ref{linqs} as our quasi-subfield polynomial, then we get complexity 
$$m!q^{d^2-dm+1}\tilde{O}(m^{5.188}2^{7.376m(m-1)}q^{4.876m(m-1)})+mq^{2d}$$
for the algorithm in \cite{Quasi}. However, since $d^2-dm+1+4.876m(m-1)>n/2$ for any $d,m$, this will not beat generic discrete log algorithms. Thus it appears that the polynomials of  Theorem \ref{gen_ab} will not lead to an ECDLP algorithm that beats generic algorithms,
although they can beat brute force algorithms.

\begin{Remark}
We  briefly discuss adding another term of small degree, for example, an $x^{q^2}$ term. Suppose we have a linearized polynomial $L=x^{q^d}-cx^{q^2}-bx^q-ax\in\mathbb{F}_{q^n}[x]$ which splits completely and with $d^2>2n$ (so $L(x)$ is a quasi-subfield polynomial). Then the algorithm of \cite{Quasi} has complexity $$m!q^{n-d(m-1)}\tilde{O}(m^{5.188}2^{7.376m(m-1)}(q^2)^{4.876m(m-1)})+mq^{2d}.$$

To beat generic discrete log algorithms, we require at least $n-d(m-1)\leq n/2$ and $2d\leq n/2$, which implies $\frac{n}{2(m-1)}\leq d\leq\frac{n}{4}$ and therefore $m\geq3$.
As an example, if we choose $q=2$ and $m=4$ then we have 
$2^{7.376m(m-1)}(q^2)^{4.876m(m-1)}\approx1.45\cdot q^{205}$ inside the $\tilde{O}$. 
This means that the overall complexity can beat generic algorithms over $\mathbb{F}_{2^n}$ (for $n$ sufficiently large). For example, a choice of $d$ around $n/5$ when $q=2$, $m=4$, would give a complexity $O(q^{0.4n})$ for $n>500$. 

To obtain an estimate for smaller field sizes we may try $m=3$, which implies that $d\approx\frac{n}{4}$. These choices would give us complexity $$q^{n/2}\tilde{O}(3^{5.188}2^{44.256}q^{58.512})+3q^{n/2}$$ which is not better than generic algorithms. One example of a linearized polynomial which splits completely and matches these choices ($q=2$, $d\approx\frac{n}{4}$)  is $L=x^{1024}+x^{4}+x\in\mathbb{F}_{2^{42}}[x]$.
\end{Remark}

\section{Conclusion and open questions}

We have provided necessary and sufficient conditions for $L=x^{q^d}-bx^q-ax\in\mathbb{F}_{q^{(d-1)d+1}}[x]$
to have all $q^d$ roots in $\mathbb{F}_{q^{(d-1)d+1}}$.

The recursive formula that we found for trinomial linearized polynomials is valid for more general linearized polynomials too: Let $L=x^{q^d}-\sum_{i=0}^{d-1}a_i x^{q^i}$.

Set $M_{l,l-d}=1$, and $M_{l,k}=0$ for $k\leq0$ and $k\ne l-d$. For $1\leq l\leq d$ and $k\geq1$, let
\begin{equation}
 M_{l,k}=M_{l,k-d}a_0^{q^{k-1}}+M_{l,k-d+1}a_1^{q^{k-1}}+\cdots+M_{l,k-1}a_{d-1}^{q^{k-1}}
 =\sum_{i=0}^{d-1} M_{l,k-d+i}a_i^{q^{k-1}}.
\end{equation}
Then $M_{l,k}$ is the $(l,d)$ entry of $A_{L,k}$. Furthermore, the $(l,j)$ entry of $A_{L,k}$ is $M_{l,k-d+j}$.

We are currently working on extending these results to this more general case, for example, to polynomials of the form $L=x^{q^d}-cx^{q^2}-bx^q-ax$.

\begin{acknowledgement}
We thank Christophe Petit and John Sheekey for helpful conversations.
\end{acknowledgement}

\begin{funding}
This research was supported by a Postgraduate Government of Ireland Scholarship from the Irish Research Council.
\end{funding}

\bibliographystyle{alpha}
\bibliography{mybib}
\end{document}